\documentclass[a4paper,11pt]{article}
\usepackage{amsmath,amsthm,amssymb,enumitem,xcolor}
\usepackage[hang]{footmisc}

\usepackage[nosort,nocompress,noadjust]{cite}

\usepackage[bookmarks=false,hyperfootnotes=false,colorlinks,linktoc=all,
    linkcolor={red!60!black},
    citecolor={blue!50!black},
    urlcolor={blue!80!black}]{hyperref}

\renewcommand{\eqref}[1]{\hyperref[#1]{(\ref{#1})}}

\newlist{enumlist}{enumerate}{2}
\setlist[enumlist,1]{labelindent=0cm,label=(\roman*),ref=(\roman*),labelwidth=4.5ex,labelsep=1ex,leftmargin=5.5ex,align=right,topsep=0.5ex,itemsep=1ex,parsep=1ex}
\setlist[enumlist,2]{labelindent=0cm,label=\alph*),ref=\arabic*,labelwidth=5ex,labelsep=0.5ex,leftmargin=5.5ex,align=left,topsep=0.5ex,itemsep=1ex,parsep=1ex}

\newlist{enumlistI}{enumerate}{1}
\setlist[enumlistI]{labelindent=0cm,label=(\Roman*),ref=(\Roman*),labelwidth=4.5ex,labelsep=1ex,leftmargin=5.5ex,align=right,topsep=0.5ex,itemsep=1ex,parsep=1ex}

\newlist{enumlistprime}{enumerate}{1}
\setlist[enumlistprime]{labelindent=0cm,label=(\roman*)$^{\prime}$,ref=(\roman*)$^{\prime}$,labelwidth=4.5ex,labelsep=1ex,leftmargin=5.5ex,align=right,topsep=0.5ex,itemsep=1ex,parsep=1ex}

\newlist{enumlistprimeprime}{enumerate}{1}
\setlist[enumlistprimeprime]{labelindent=0cm,label=(\roman*)$^{\prime\prime}$,ref=(\roman*)$^{\prime\prime}$,labelwidth=4.5ex,labelsep=1ex,leftmargin=5.5ex,align=right,topsep=0.5ex,itemsep=1ex,parsep=1ex}

\newlist{enumlistprimeprimeprime}{enumerate}{1}
\setlist[enumlistprimeprimeprime]{labelindent=0cm,label=(\roman*)$^{\prime\prime\prime}$,ref=(\roman*)$^{\prime\prime\prime}$,labelwidth=4.5ex,labelsep=1ex,leftmargin=5.5ex,align=right,topsep=0.5ex,itemsep=1ex,parsep=1ex}

\newlist{enuma}{enumerate}{1}
\setlist[enuma]{labelindent=0cm,label=(\alph*),ref=(\alph*),labelwidth=4.5ex,labelsep=1ex,leftmargin=5.5ex,align=right,topsep=0.5ex,itemsep=1ex,parsep=1ex}

\newlist{itemlist}{itemize}{1}
\setlist[itemlist]{labelindent=0cm,label=$\bullet$,labelwidth=2.5ex,labelsep=0.5ex,leftmargin=3ex,align=left,topsep=0.5ex,itemsep=1ex,parsep=1ex}

{\theoremstyle{definition}\newtheorem{definitiona}{Definition}
\newtheorem{remarka}[definitiona]{Remark}
\newtheorem{examplea}[definitiona]{Example}}

\newtheorem{propositiona}[definitiona]{Proposition}
\newtheorem{lemmaa}[definitiona]{Lemma}
\newtheorem{theorema}[definitiona]{Theorem}
\newtheorem{corollarya}[definitiona]{Corollary}

\newtheorem{letterthma}{Theorem}
\renewcommand{\theletterthma}{\Alph{letterthma}}

\newtheorem{letterpropa}[letterthma]{Proposition}

{\theoremstyle{definition}
\newtheorem{letterremarka}[letterthma]{Remark}
\renewcommand{\theletterremarka}{\Alph{letterthma}}
}

\newenvironment{lemma}[1][]{\begin{lemmaa}[#1]\setlist*[enumlist,1]{label=(\roman*),ref=\thelemmaa(\roman*)}\setlist*[enuma]{label=(\alph*),ref=\thelemmaa(\alph*)}}{\end{lemmaa}}

\newenvironment{letterthm}[1][]{\begin{letterthma}[#1]\setlist*[enumlist,1]{label=(\roman*),ref=\theletterthma(\roman*)}\setlist*[enuma]{label=(\alph*),ref=\theletterthma(\alph*)}}{\end{letterthma}}
\newenvironment{letterprop}[1][]{\begin{letterpropa}[#1]\setlist*[enumlist,1]{label=(\roman*),ref=\theletterthma(\roman*)}\setlist*[enuma]{label=(\alph*),ref=\theletterthma(\alph*)}}{\end{letterpropa}}

\newcommand{\C}{\mathbb{C}}
\newcommand{\Z}{\mathbb{Z}}
\newcommand{\N}{\mathbb{N}}
\newcommand{\R}{\mathbb{R}}

\newcommand{\T}{\mathbb{T}}
\newcommand{\Q}{\mathbb{Q}}

\newcommand{\cZ}{\mathcal{Z}}
\newcommand{\cG}{\mathcal{G}}

\newcommand{\cU}{\mathcal{U}}

\newcommand{\al}{\alpha}
\newcommand{\be}{\beta}
\newcommand{\vphi}{\varphi}
\newcommand{\om}{\omega}
\newcommand{\Om}{\Omega}

\newcommand{\Ad}{\operatorname{Ad}}

\newcommand{\Aut}{\operatorname{Aut}}

\newcommand{\SL}{\operatorname{SL}}

\newcommand{\alhat}{\widehat{\alpha}}

\newcommand{\Shat}{\widehat{S}}
\newcommand{\That}{\widehat{T}}

\newcommand{\ot}{\otimes}
\newcommand{\id}{\mathord{\text{\rm id}}}
\newcommand{\ovt}{\mathbin{\overline{\otimes}}}
\newcommand{\actson}{\curvearrowright}

\begin{document}

\begin{center}
{\boldmath\LARGE\bf Factoriality of twisted locally compact\vspace{0.5ex}\\
group von Neumann algebras}

\vspace{1ex}

{\sc by Stefaan Vaes\footnote{KU~Leuven, Department of Mathematics, Leuven (Belgium), stefaan.vaes@kuleuven.be\\ Supported by FWO research project G016325N of the Research Foundation Flanders and by Methusalem grant METH/21/03 –- long term structural funding of the Flemish Government.}}
\end{center}

\begin{abstract}\noindent
In this short note, we construct an exotic example of a locally compact group $\cG$ with a Borel $2$-cocycle $\om$ such that the non-twisted group von Neumann algebra $L(\cG)$ is a factor, while the twisted group von Neumann algebra $L_\om(\cG)$ has a diffuse center.
\end{abstract}

\section{Introduction}

When $\Gamma$ is a discrete group, it was already proven by Murray and von Neumann that the group von Neumann algebra $L(\Gamma)$ is a factor if and only if every nontrivial conjugacy class of $\Gamma$ is infinite. In that case, every $2$-cocycle twisted group von Neumann algebra $L_\om(\Gamma)$ automatically is a factor as well. Note that in general, \cite{Kle61} provides a criterion for factoriality of arbitrary twisted group von Neumann algebras $L_\om(\Gamma)$ of discrete groups.

For locally compact groups $\cG$, there is no group-theoretic criterion for the factoriality of its group von Neumann algebra $L(\cG)$. In this short note, we give examples where $L(\cG)$ is a factor, while certain $2$-cocycle twists $L_\om(\cG)$ have a diffuse center. These counterintuitive examples provide strong evidence that it is very unlikely that one can give an intrinsic characterization for the factoriality of twisted group von Neumann algebras in the locally compact case.

\begin{letterthm}\label{thm.main}
Let $p$ be an odd prime number and define the locally compact group $G$ as the semidirect product $G = \Q_p^2 \rtimes \SL_2(\Q)$, with compact open subgroup $K = \Z_p^2$. Define $\cG$ as the restricted product $\cG = \prod_{k \in \N}' (G,K)$. Then $\cG$ admits a Borel $2$-cocycle $\om$ such that $L(\cG)$ is a factor and $L_\om(\cG)$ has a diffuse center.
\end{letterthm}

Our initial motivation to search for the examples in Theorem \ref{thm.main} was however different. Recently in \cite{DCK24}, a theory of braided tensor products of von Neumann algebras was developed, where the braiding comes from actions of locally compact quantum groups. There then arose in \cite{DCK24} the natural question if the braided tensor product of two factors is always a factor. The answer turns out to be no, and to give a counterexample, it suffices to construct an action $\cG \actson^\al A$ of a locally compact group $\cG$ on a von Neumann algebra $A$ such that both $L(\cG)$ and $A$ are factors, while the crossed product $A \rtimes_\al \cG$ is not a factor. By Theorem \ref{thm.main}, such actions indeed exist and this is used in the proof of \cite[Corollary 8.5]{DCK24}. Such actions even exist with $A = B(K)$ and $\al$ an inner action.

These phenomena only happen in the nondiscrete case, as we show in the second result of this note.

\begin{letterprop}\label{prop.main-cor}
Let $G \actson^\al A$ be any continuous action of a locally compact group $G$ on a von Neumann algebra $A$. Assume that both $L(G)$ and $A$ are factors. If $G$ is discrete, then $A \rtimes_\alpha G$ is a factor. If $G$ is not discrete, this does not always hold.
\end{letterprop}

\section{Proof of Theorem \ref{thm.main}}

We deduce Theorem \ref{thm.main} from a series of lemmas. Lemma \ref{lem.generic-1} provides a generic construction of an action $\cG \actson M$ of a locally compact group $\cG$ on a von Neumann algebra $M$ such that both $L(\cG)$ and $M \rtimes \cG$ have other crossed product descriptions that will allow us to give examples where $L(\cG)$ is a factor, while $M \rtimes \cG$ has a diffuse center. Moreover, as we prove in Lemma \ref{lem.generic-2}, we can give such examples with $M \cong B(K)$ for a separable Hilbert space $K$, so that $M \rtimes \cG \cong L_\om(\cG) \ovt B(K)$ for some Borel $2$-cocycle $\om \in Z^2(\cG,\T)$, by Lemma \ref{lem.twisted}.

Throughout this section, we abbreviate \emph{locally compact second countable} as \emph{lcsc}.

\begin{lemma}\label{lem.generic-1}
Let $G$ be a lcsc group and $G \actson^\al S$ a continuous action of $G$ by automorphisms of a lcsc abelian group $S$. Consider the action $G \actson^{\alhat} \Shat$ on the Pontryagin dual given by $\alhat_g(\om) = \om \circ \al_g^{-1}$. Let $\be$ be a continuous action of the semidirect product $S \rtimes_\al G$ on a von Neumann algebra $A$. Define the semidirect product group $\cG = \Shat \rtimes_{\alhat} G$ and crossed product von Neumann algebra $M = A \rtimes_\be S$.

Then, $\cG$ admits a continuous action on $M$ such that $M \rtimes \cG \cong (A \rtimes_\be G) \ovt B(L^2(S))$. Also, $L(\cG) \cong L^\infty(S) \rtimes_\al G$.
\end{lemma}
\begin{proof}
We denote by $(u_x)_{x \in S}$ the natural unitary operators in $M = A \rtimes_\be S$. Consider the dual action
$$\Shat \actson^\gamma M : \gamma_\om(a u_x) = \om(x) \, a u_x \quad\text{for all $a \in A$, $x \in S$, $\om \in \Shat$.}$$
We also consider the action
$$G \actson^\gamma M : \gamma_g(a u_x) = \be_g(a) \, u_{\al_g(x)} \quad\text{for all $a \in A$, $x \in S$, $g \in G$,}$$
which is well-defined because $\be_g \in \Aut A$ is an automorphism that conjugates the action $(\be_x)_{x \in S}$ and the action $(\be_{\al_g(x)})_{x \in S}$.

By construction, $\gamma_g \circ \gamma_\om = \gamma_{\alhat_g(\om)} \circ \gamma_g$, so that both actions combine into a continuous action $\cG \actson^\gamma M$. The crossed product $M \rtimes_\gamma \Shat$ is generated by $M$ and unitaries $(v_\om)_{\om \in \Shat}$. By the Takesaki duality theorem, we find a $*$-isomorphism
$$\pi : M \rtimes_\gamma \Shat \to A \ovt B(L^2(S)) : \begin{cases} \pi(a) \in A \ovt L^\infty(S) : \pi(a)(y) = \be_y(a) \quad\text{for $a \in A$, $y \in S$,}\\
\pi(u_x) = 1 \ot \rho_x \quad\text{for $x \in S$, where $(\rho_x \xi)(y) = \xi(x+y)$,}\\
\pi(v_\om) = 1 \ot \overline{\om} \in 1 \ot L^\infty(S) \quad\text{for $\om \in \Shat$.}\end{cases}$$
We have the natural action $G \actson^\zeta M \rtimes_\gamma \Shat$ given by $\zeta_g(d) = \gamma_g(d)$ for all $g \in G$, $d \in M$ and $\zeta_g(v_\om) = v_{\alhat_g(\om)}$ for all $g \in G$, $\om \in \Shat$. By construction, $M \rtimes_\gamma \cG \cong (M \rtimes_\gamma \Shat) \rtimes_\zeta G$.

Since each $\al_g$ is an automorphism of the lcsc group $S$, it scales the Haar measure of $S$ by a constant $\eta(g)$, so that
$$W_g \in \cU(L^2(S)) : (W_g \xi)(y) = \eta(g)^{1/2} \xi(\al_g^{-1}(y)) \quad\text{for $g \in G$, $\xi \in L^2(S)$, $y \in S$,}$$
defines a unitary representation $(W_g)_{g \in G}$ of $G$ on $L^2(S)$.

By construction, $\pi \circ \zeta_g \circ \pi^{-1} = \be_g \ot \Ad W_g$ for all $g \in G$. Since $(1 \ot W_g^*)_{g \in G}$ is a $1$-cocycle for the action $(\be_g \ot \Ad W_g)_{g \in G}$, it follows that
$$(M \rtimes_\gamma \Shat) \rtimes_\zeta G \cong (A \rtimes_\be G) \ovt B(L^2(S)) \; .$$
Finally note that $L(\Shat) \cong L^\infty(S)$ through the Fourier transform, so that $L(\cG) = L(\Shat \rtimes_{\alhat} G) \cong L^\infty(S) \rtimes_\al G$.
\end{proof}

In Lemma \ref{lem.generic-2} below, we provide a more specific variant of the construction in Lemma \ref{lem.generic-1}, in which $M \cong B(K)$, where $K$ is a separable Hilbert space. By the following elementary lemma, which is certainly well-known, it then follows that $M \rtimes \cG \cong L_\om(\cG)$ for some Borel $2$-cocycle $\om \in Z^2(\cG,\T)$. For completeness, we provide a detailed proof.

\begin{lemma}\label{lem.twisted}
Let $G$ be a lcsc group and $G \actson^\al M$ a continuous action on a factor $M$ with separable predual. Assume that $\al_g$ is an inner automorphism for every $g \in G$. There then exists a Borel map $\pi : G \to \cU(M)$ such that $\al_g = \Ad \pi(g)$ for all $g \in G$. Then, $\pi(g)\pi(h) = \Om(g,h) \pi(gh)$ for all $g,h \in G$, where $\Om \in Z^2(G,\T)$ is a Borel $2$-cocycle. There is a unique $*$-isomorphism
$$\theta : A \ovt L_{\overline{\Om}}(G) \to A \rtimes_\al G \quad\text{satisfying}\quad \theta(a \ot \lambda_{\overline{\Om}}(g)) = a \pi(g)^* u_g \quad\text{for all $a \in A$ and $g \in G$,}$$
where $\lambda_{\overline{\Om}}$ is the regular $\overline{\Om}$-representation on $L^2(G)$ that generates $L_{\overline{\Om}}(G)$.
\end{lemma}
\begin{proof}
Define the closed subgroup $P$ of the Polish group $G \times \cU(M)$ as $P = \{(g,u) \mid \al_g = \Ad u\}$. By assumption, the homomorphism $(g,u) \mapsto g$ restricts to a surjective continuous homomorphism $P \to G$ whose kernel $\{e\} \times \T \cdot 1$ is compact. By e.g.\ \cite[Theorem 1.2]{Sri80}, we can choose a Borel lift $G \to P : g \mapsto (g,\pi(g))$. Since $M$ is a factor and $\al_g \circ \al_h = \al_{gh}$, we find that $\pi(g)\pi(h) = \Om(g,h) \pi(gh)$ for all $g,h \in G$, where $\Om \in Z^2(G,\T)$ is a Borel $2$-cocycle.

Assume that $A \subset B(H)$, equip $G$ with a left Haar measure and realize $A \rtimes_\al G$ as the von Neumann algebra acting on $L^2(G,H)$ generated by
$$(a \cdot \xi)(h) = \al_{h^{-1}}(a) \xi(h) \quad\text{and}\quad (u_g \cdot \xi)(h) = \xi(g^{-1}h) \quad\text{for all $a \in A$, $g,h \in G$.}$$
Define the unitary operator $V$ on $L^2(G,H)$ by $(V \xi)(h) = \pi(h) \xi(h)$. Identifying $L^2(G,H) = H \ot L^2(G)$, one computes that
$$V a u_g V^* = a \pi(g) \otimes \lambda_{\overline{\Om}}(g) \quad\text{for all $a \in A$, $g \in G$.}$$
So, we can define $\theta = \Ad V^*$.
\end{proof}

For the following lemma, recall that for an abelian group $S$, a map $\Om : S \times S \to \T$ is called a bicharacter if $\Om(\cdot,y)$ and $\Om(x,\cdot)$ are homomorphisms from $S$ to $\T$ for all $x,y \in S$.

\begin{lemma}\label{lem.generic-2}
Let $G \actson^\al S$ be a continuous action of a lcsc group by automorphisms of a lcsc abelian group $S$. Let $\Om : S \times S \to \T$ be a continuous bicharacter satisfying $\Om(\al_g(x),\al_g(y)) = \Om(x,y)$ for all $x,y \in S$, $g \in G$.

Assume that there exists a lcsc abelian group $T$ and an isomorphism $\theta : T \times \That \to S$ such that $T \mapsto T : y \mapsto 2 y$ is a homeomorphism of $T$ onto $T$, and such that
$$\Om(\theta(y,\vphi),\theta(y',\vphi')) = \vphi'(y) \, \overline{\vphi(y')} \quad\text{for all $y,y' \in T$, $\vphi,\vphi' \in \That$.}$$
Then the semidirect product group $\cG = \Shat \rtimes_{\alhat} G$ admits a Borel $2$-cocycle $\om \in Z^2(\cG,\T)$ such that
$$\cZ(L_\om(\cG)) \cong \cZ(L(G)) \quad\text{while}\quad \cZ(L(\cG)) \cong \cZ(L^\infty(S) \rtimes_\al G) \; .$$
\end{lemma}
\begin{proof}
The formula $\eta((x,g),(y,h)) = \Om(x,\al_g(y))$ defines a continuous $2$-cocycle on the semidirect product group $S \rtimes_\al G$. Choose an $\overline{\eta}$-projective representation $\pi : S \rtimes_\al G \to \cU(H)$ on a separable Hilbert space $H$. Write $A = B(H)$ and define the inner action $\be_g = \Ad \pi(g)$ of $S \rtimes_\al G$ on $A$. We apply Lemma \ref{lem.generic-1}. We thus find a continuous action $\gamma$ of $\cG$ on $M = A \rtimes_\be S$ such that
\begin{equation}\label{eq.equal-centers}
\cZ(M \rtimes_\gamma \cG) \cong \cZ(A \rtimes_\be G) \; .
\end{equation}
Below we prove that $M$ is a type~I factor $B(K)$. So, $\gamma$ is an inner action that, by Lemma \ref{lem.twisted}, gives rise to a Borel $2$-cocycle $\om$ on $\cG$ such that
$$M \rtimes_\gamma \cG \cong B(K) \ovt L_\om(\cG) \; .$$
By Lemma \ref{lem.twisted} and because the $2$-cocycle $\eta$ is trivial on $G \times G$, we have that $A \rtimes_\be G \cong B(K) \ovt L(G)$. It then follows from \eqref{eq.equal-centers} that $\cZ(L_\om(\cG)) \cong \cZ(L(G))$. In Lemma \ref{lem.generic-1}, we already observed that $L(\cG) \cong L^\infty(S) \rtimes_\al G$.

It remains to prove that $M$ is a type I factor. Since $\be$ is an inner action and $A = B(H)$, it follows from Lemma \ref{lem.twisted} that $M \cong B(H) \ovt L_\eta(S) = B(H) \ovt L_\Om(S)$. Consider the action $T \actson^\zeta L^\infty(T) : (\zeta_y(F))(z) = F(2y+z)$. Since we assumed that $y \mapsto 2y$ is an isomorphism of $T$ onto $T$, we get that $L^\infty(T) \rtimes_\zeta T \cong B(L^2(T))$. Denote by $(\lambda_\Om(x))_{x \in S}$ the canonical generating unitaries of $L_\Om(S)$. Denote the neutral elements of $T$ and $\That$ by $0$ and $1$, resp. By definition,
$$\lambda_\Om(\theta(y,1)) \, \lambda_\Om(\theta(0,\vphi)) = \vphi(y) \, \lambda_\Om(\theta(y,\vphi)) = \vphi(y)^2 \, \lambda_\Om(\theta(0,\vphi)) \, \lambda_\Om(\theta(y,1)) \; .$$
There is thus a unique $*$-isomorphism $\Theta : L^\infty(T) \rtimes_\zeta T \to L_\Om(S)$ satisfying
$$\Theta(\vphi) = \lambda_\Om(0,\vphi) \quad\text{and}\quad \Theta(\lambda_y) = \lambda_\Om(y,1) \quad\text{for all $\vphi \in \That \subset L^\infty(T)$ and $y \in T$.}$$
It follows that $L_\Om(S) \cong B(L^2(T)$, so that $M$ is a type I factor.
\end{proof}

To prove Theorem \ref{thm.main}, it suffices to give examples satisfying the assumptions of Lemma \ref{lem.generic-2} such that the action $G \actson^\al S$ is essentially free and ergodic, while the center of $L(G)$ is diffuse. For this, we use the following lemma.

\begin{lemma}\label{lem.generic-3}
Let $R$ be a lcsc ring that is not discrete and that has no zero divisors. Assume that $(R,+)$ admits a continuous character $\psi : R \to \T$ such that the bicharacter $(x,y) \mapsto \psi(xy)$ implements an isomorphism $R \cong \widehat{R}$.

Let $\Gamma$ be a countable subgroup of $\SL_2(R)$ and consider the canonical action $\Gamma \actson^\al R^2$. Define the bicharacter
$$\Om : R^2 \times R^2 \to \T : \Om((x,x'),(y,y')) = \psi(xy' - x' y) \; .$$
Assume that the action $\Gamma \actson R^2$ is ergodic.

Then the semidirect product group $\cG = {\widehat{R}}^2 \rtimes_{\alhat} \Gamma$ has a factorial group von Neumann algebra $L(\cG)$ and admits a Borel $2$-cocycle $\om \in Z^2(\cG,\T)$ such that $\cZ(L_\om(\cG)) \cong \cZ(L(\Gamma))$.
\end{lemma}
\begin{proof}
By a direct computation, $\Om(\al_g(x,x'),\al_g(y,y')) = \Om((x,x'),(y,y'))$ for all $g \in \Gamma \subset \SL_2(R)$.

Note that the action $\Gamma \actson R^2$ is essentially free. To prove this statement, since $\Gamma$ is countable, it suffices to prove that for every matrix $A \in \SL_2(R)$ with $A \neq I_2$, the set $\{x \in R^2 \mid A \cdot x = x\}$ has Haar measure zero. Taking a nonzero row of $A-I_2$, it suffices to prove that for all $a,b \in R$ that are not both equal to zero, the set $D = \{(x,x') \in R^2 \mid ax+bx' = 0\}$ has Haar measure zero. Assume that $a \neq 0$. For every fixed $x' \in R$, because $R$ has no zero divisors, there is at most one $x \in R$ such that $(x,x') \in D$. Since $R$ is not discrete, singletons have Haar measure zero and the result follows from the Fubini theorem.

Since we assumed that $\Gamma \actson R^2$ is ergodic, it follows that $L^\infty(R^2) \rtimes_\al \Gamma$ is a factor. So, $L(\cG)$ is a factor.

By assumption, we can view $R^2$ as $R \times \widehat{R}$ such that $\Om$ is of the form required by Lemma \ref{lem.generic-2}. Since $2 \in R$ is invertible, $x \mapsto 2 x$ is a homeomorphism of $R$ onto $R$. The conclusion thus follows from Lemma \ref{lem.generic-2}.
\end{proof}

For the following lemma, recall the notion of a restricted direct product of locally compact groups: given a sequence of locally compact groups $G_k$ with compact open subgroups $K_k$, the set $\prod'_{k \in \N} (G_k,K_k)$ is defined as the set of $(g_k)_{k \in \N}$ with $g_k \in G_k$ for all $k \in \N$ and $g_k \in K_k$ for $k$ sufficiently large. With pointwise operations and the product topology, this is a locally compact group.

\begin{lemma}\label{lem.examples}
In the following cases, all assumptions of Lemma \ref{lem.generic-3} are satisfied.
\begin{enumlist}
\item $R = \R$ with $\psi(x) = \exp(i x)$ and $\Gamma = \SL_2(\Z)$ or $\Gamma = \SL_2(\Q)$. In both cases, $\cZ(L(\Gamma))$ is $2$-dimensional.
\item $R = \R^n$ (with componentwise multiplication) with $\psi(x) = \exp(i (x_1 + \cdots + x_n))$ and $\Gamma = \SL_2(\Z)^n$ or $\Gamma = \SL_2(\Q)^n$. In both cases, $\cZ(L(\Gamma))$ is $2^n$-dimensional.
\item\label{lem.examples.3} For an odd prime $p$, take
\begin{alignat*}{2}
R = {\prod_{k \in \N}}' (\Q_p,\Z_p) \quad &\text{with}\quad &&\psi(x) = \psi_0\Bigl(\sum_{k \in \N} (x_k + \Z_p)\Bigr) \\
&\text{and}\quad &&\psi_0 : \Q_p / \Z_p \to \T : \psi_0(p^{-n} r + \Z_p) = \exp(2\pi i p^{-n} r) \\
&&&\text{for all $n \in \N$, $r \in \Z$,}
\end{alignat*}
and $\Gamma = \SL_2(\Q)^{(\N)}$. Now the center of $L(\Gamma)$ is diffuse.
\end{enumlist}
\end{lemma}
\begin{proof}
In each of the three cases, we only need to prove the ergodicity of $\Gamma \actson R^2$ and determine the center of $L(\Gamma)$.

(i) To prove the ergodicity of $\SL_2(\Z) \actson \R^2$, we define the closed subgroup $P \subset \SL_2(\R)$ of matrices of the form $\bigl(\begin{smallmatrix} 1 & x \\ 0 & 1\end{smallmatrix}\bigr)$. Identifying $\R^2 \setminus \{(0,0)\}$ with $\SL_2(\R) / P$, we get that
$$L^\infty(\R^2)^{\SL_2(\Z)} \cong L^\infty(\SL_2(\R) / \SL_2(\Z))^P \; .$$
The right hand side equals $\C 1$ by Moore's ergodicity theorem, because $\SL_2(\Z)$ is a lattice in $\SL_2(\R)$ and $P$ is noncompact.

The action $\SL_2(\R) \actson \R^2$ is essentially transitive and thus certainly ergodic. Then also the action of the dense subgroup $\SL_2(\Q)$ on $\R^2$ is ergodic.

Both in $\SL_2(\Z)$ and in $\SL_2(\Q)$, every element different from $I_2$ and $-I_2$ has an infinite conjugacy class, while $\pm I_2$ are central. So, the center of $L(\SL_2(\Z))$ and of $L(\SL_2(\Q))$ is $2$-dimensional.

(ii) When $\Gamma \actson \R^2$ is ergodic, also the product action of $\Gamma \times \cdots \times \Gamma$ on $\R^2 \times \cdots \times \R^2$ is ergodic. Point (ii) thus follows from point (i).

(iii) The action $\SL_2(\Q_p) \actson \Q_p^2$ is essentially transitive and thus, ergodic. Since $\SL_2(\Q)$ is dense in $\SL_2(\Q_p)$, also the action $\SL_2(\Q) \actson \Q_p^2$ is ergodic. For every $n \in \N$, consider the $n$-fold products $\Gamma_n = \SL_2(\Q)^n$ and $S_n = (\Q_p^2)^n$. Then also $\Gamma_n \actson S_n$ is ergodic. Write $K_n = \prod_{k = n+1}^\infty \Z_p^2$ and view $X_n := S_n \times K_n$ as a subset of $R^2$. By definition, $X_n$ is an increasing sequence of subsets of $R^2$ whose union equals $R^2$.

Take a $\Gamma$-invariant element $F \in L^\infty(R^2)$. For every $n \in \N$, the restriction of $F_n$ to $X_n$ is $(\Gamma_n \times \id)$-invariant. By the ergodicity of $\Gamma_n \actson S_n$, we find $H_n \in L^\infty(K_n)$ such that $F_n = 1 \ot H_n$ a.e. Fix $n \in \N$. For every $m > n$, the function $F_n$ is the restriction of the function $F_m$. It follows that $H_n = 1^{m-n} \ot H_m$ a.e.\ for every $m > n$. This implies that $H_n$ is essentially constant. So, $F_n$ is essentially constant for every $n$. Since $\bigcup_n X_n = R^2$, all these constants must be the same and $F$ is essentially constant.

Now the infinite subgroup $\{\pm I_2\}^{(\N)}$ lies in the center of $\Gamma$. So the center of $L(\Gamma)$ is diffuse.
\end{proof}

\begin{proof}[Proof of Theorem \ref{thm.main}]
It suffices to combine Lemma \ref{lem.generic-3} with Lemma \ref{lem.examples.3}, and to observe that the actions of $\SL(2,R)$ on $R^2$ and ${\widehat{R}}^2$ are isomorphic.
\end{proof}

\section{Proof of Proposition \ref{prop.main-cor}}

\begin{proof}[Proof of Proposition \ref{prop.main-cor}]
Assume that $G$ is discrete and that both $L(G)$ and $A$ are factors. We have to prove that $A \rtimes_\al G$ is a factor. Define $G_0 < G$ as the normal subgroup of $g \in G$ for which $\al_g$ is an inner automorphism of $A$. For every $g \in G_0$, choose a unitary $v_g \in A$ such that $\al_g = \Ad v_g$. We make this choice such that $v_e = 1$. Since $A$ is a factor, $v_g$ is uniquely determined up to multiplication by a scalar of modulus $1$. Thus, $v_g v_h = \om(g,h) v_{gh}$ for all $g,h \in G_0$, where $\om \in Z^2(G_0,\T)$ is a $2$-cocycle.

Write $M = A \rtimes_\al G$, generated by $A$ and unitaries $(u_g)_{g \in G}$. We also define $B = A \rtimes_\al G_0$. We denote by $E_B : M \to B$ the canonical faithful normal conditional expectation satisfying $E_B(u_g) = 0$ for all $g \in G \setminus G_0$. We similarly denote by $E_A : M \to A$ the canonical faithful normal conditional expectation satisfying $E_A(u_g) = 0$ for all $g \in G \setminus \{e\}$.

We start by proving that $A' \cap M \subset B$. Fix $d \in A' \cap M$. Since $A \subset B$, also $E_B(d) \in A' \cap M$. Replacing $d$ by $d - E_B(d)$, we may thus assume that $E_B(d) = 0$ and prove that $d = 0$. Fix $g \in G \setminus G_0$. Then for all $a \in A$,
$$a E_A(d u_g^*) = E_A(a d u_g^*) = E_A(d a u_g^*) = E_A(d u_g^* \al_g(a)) = E_A(d u_g^*) \al_g(a) \; .$$
Since $A$ is a factor and $\al_g$ is an outer automorphism, it follows that $E_A(d u_g^*) = 0$ for all $g \in G \setminus G_0$. When $g \in G_0$, since $E_A = E_A \circ E_B$, we have that
$$E_A(d u_g^*) = E_A(E_B(d u_g^*)) = E_A(E_B(d) u_g^*) = 0 \; .$$
We thus conclude that $E_A(d u_g^*) = 0$ for all $g \in G$. Then also $E_A(d u_g^* a) = E_A(d u_g^*) a = 0$ for all $g \in G$ and $a \in A$. It follows that $E_A(d b) = 0$ for all $b \in M$. So, $E_A(dd^*) = 0$ and it follows that $d = 0$.

Denote by $(W_h)_{h \in G_0}$ the canonical unitaries in $L_{\overline{\om}}(G_0)$. In the case of discrete groups, the proof of Lemma \ref{lem.twisted} makes no use of the second countability assumption. So, there is a unique $*$-isomorphism $\theta : A \ovt L_{\overline{\om}}(G_0) \to B$ satisfying $\theta(a \ot W_h) = a v_h^* u_h$ for all $a \in A$ and $h \in G_0$.

Note that $G_0$ is a normal subgroup of $G$ and that, for every $g \in G$ and $h \in G_0$,
$$\al_{ghg^{-1}} = \al_g \circ \al_h \circ \al_g^{-1} = \al_g \circ (\Ad v_h) \circ \al_g^{-1} = \Ad \al_g(v_h) \; .$$
It follows that $v_{ghg^{-1}}$ must be a multiple of $\al_g(v_h)$. Define $\gamma_g(h) \in \T$ such that $\al_g(v_h)=\gamma_g(h) v_{ghg^{-1}}$. Then for all $g \in G$ and $h \in G_0$, we have that
\begin{align*}
u_g \, \theta(a \ot W_h) \, u_g^* &= u_g \, a v_h^* u_h \, u_g^* = \al_g(a) \al_g(v_h)^* u_{ghg^{-1}} = \overline{\gamma_g(h)} \, \al_g(a) v_{ghg^{-1}}^* u_{ghg^{-1}} \\
&= \theta(\al_g(a) \ot \overline{\gamma_g(h)}\, W_{ghg^{-1}}) \; .
\end{align*}
It follows that there is a unique action $\zeta$ of $G$ on $L_{\overline{\om}}(G_0)$ satisfying
$$\zeta_g(a \ot W_h) = \al_g(a) \ot \overline{\gamma_g(h)} \, W_{ghg^{-1}} \quad\text{and}\quad \theta \circ \zeta_g = (\Ad u_g) \circ \theta \quad\text{for all $g \in G$, $h \in G_0$.}$$

Now take $d \in \cZ(M)$. We proved above that $A' \cap M \subset B$. So, $d \in \cZ(B)$. Using the isomorphism $\theta$ and the assumption that $A$ is a factor, it follows that $d = \theta(1 \ot b)$ for some $b \in L_{\overline{\om}}(G_0)$ satisfying $\zeta_g(b) = b$ for all $g \in G$. Since $L(G)$ is a factor, the group $G$ has infinite conjugacy classes. A fortiori, the $G$-conjugacy class of every $h \in G_0 \setminus \{e\}$ is infinite. It follows that the only $(\zeta_g)_{g \in G}$-invariant elements of $L_{\overline{\om}}(G_0)$ are the scalars. We have proven that $b \in \C 1$ and thus $d \in \C 1$.

For the following reason, the proposition does not hold without the discreteness assumption. By Theorem \ref{thm.main} and Lemma \ref{lem.twisted}, we know that there exist locally compact groups $\cG$ and inner actions $\cG \actson^\al B(K)$ such that $L(\cG)$ is a factor, while $B(K) \rtimes_\al \cG$ has a diffuse center.
\end{proof}

\end{document}